\documentclass[11pt,twoside,a4paper]{article}
\usepackage{amsmath,amssymb,amsthm}

\theoremstyle{plain}
\newtheorem{theorem}{Theorem}
\newtheorem{lemma}{Lemma}
\newtheorem{corollary}{Corollary}

\theoremstyle{definition}
\newtheorem{definition}{Definition}
\newtheorem{example}{Example}
\newtheorem{remark}{Remark}

\theoremstyle{plain}
\newtoks\thehProclaim
\newtheorem*{Proclaim}{\the\thehProclaim}

\theoremstyle{definition}
\newtoks{\thehRemark}
\newtheorem*{Remark}{\the\thehRemark}

\renewcommand{\leq}{\leqslant}
\renewcommand{\geq}{\geqslant}

\DeclareMathOperator{\diam}{diam}
\DeclareMathOperator{\dist}{dist}

\DeclareMathOperator{\V}{Vol}
\DeclareMathOperator{\Ir}{Ir}
\newcommand*{\Hn}{\mu}
\newcommand*{\capU}[1]{\partial_\Gamma^1{#1}}
\newcommand*{\capD}[1]{\partial_\Gamma^2{#1}}

\newcommand*{\capN}[2]{\partial_\Gamma^{#1}{#2}}

\newcommand*{\trU}[1]{{{#1}^*}}
\newcommand*{\trD}[1]{{{#1}_*}}
\newcommand*{\trN}[2]{{{#1}^{{#2}}}}
\newcommand*{\tracenorm}[1]{||#1||_\Gamma}

\newcommand{\BVO}{{BV (\Omega)}}
\newcommand{\D}{\partial}
\newcommand{\Rn}{\mathbb{R}^n}
\newcommand{\R}{\mathbb{R}}
\newcommand{\ep}{\epsilon}

\begin{document}
\title{Trace of $BV$-functions on irregular subsets}
\author{Yu.D.~Burago
and N.N.~Kosovskiy
\footnote{The paper partially supported by the grant RFBR 08-01-00079a}}
\date{}
\maketitle 
\section{Introduction.} The purpose of the paper is to generalize all the main results
on  boundary trace
of the book \cite{M}, Chapter 6,  to a  wider class of sets.
This chapter is an extended
version of the earlier publication \cite{BM}. Our paper is an extended and completed
version of our publication  \cite{BK}, where some results were presented
without proofs or in a weaken form.
In \cite{BM}, \cite{M}, boundary trace was defined for regions
$\Omega$ with finite perimeter (in the sense of Cacioppoli--De Giorgi)
and the main results about trace were obtained under an additional assumption that
normals in the sense of Federer exist almost everywhere on
 $\D\Omega$. Instead of that, here we suppose that
$\D\Omega$ is a countably (n-1)-rectifiable set, which is a more general
condition.
Readers can get acquainted with the theory of sets of finite perimeter and
$BV$ functions in the books  \cite{M},
 \cite{Dzh}, \cite{Ziem}.

The  analytical tools we use  are basically the same
as  in  \cite{BM}, \cite{M}. Relations between isoperimetric inequalities and
integral inequalities (of Sobolev embedding
theorems type) play an essential role. First these relations were discovered by V.~Maz'ya
\cite{M1}.
Almost all results formulated below are valid not only for regions in  $\mathbb{R}^n$
but for regions on $C^1$-smooth $n$-dimensional manifolds as well. This becomes clear
from Corollary \ref{cor_generalization}.

In fact, deep knowledge in geometric measure
theory, in particular, in rectifiable currents  is not necessary.
All necessary (very restricted) information from this theory are given below.

Let us explain the reason for our results to generalize  those  in \cite{BM}.
It is known that the boundary  $\D E$ of a set  $E\subset\Rn$ with a finite perimeter
consists of two parts. One of them, so caller reduced boundary $\D^*E$,
consists of all points at which normals in the sense of Federer exist.
It is known that this part is a countably $(n-1)$-rectifiable  set. The perimeter $P(E)$
of a set $E$ equals $H_{n-1}(\D^*E)$, where $H_k$ is $k$-dimensional Hausdorff measure.
So the requirement, that the normals in the sense of Federer exist a.e. on  $\D^*\Omega$
is equivalent to the condition that $\D\Omega=\D^*\Omega$. For sure, all sets
are considered up to sets of $(n-1)$-dimensional Hausdorff measure zero.

In general, $\D E\setminus\D^* E$ consist of two parts, a countably $(n-1)$-rectifiable set
and so called completely unrectified (irregular) set $\Ir(E)$. The latter may have either
finite,
or infinite $(n-1)$-dimensional Hausdorff measure.
The assumption that  $\D\Omega$ is a countably $(n-1)$-rectifiable  set means that
the set  $\Ir$ is empty. However even in this case the countably rectifiable set
$\D\Omega$ can be essentially vaster than
 $\D^* E$.

Let us explain this situation by the following example. Consider  an open
disk  in a plane with a sequence of intervals $I_i$
removed. Suppose that the union of these intervals is closed.
The results of \cite{BM} on boundary traces
are not applicable to such a region $\Omega$ (the intervals do not belong to the reduced
boundary) but the boundary of $\Omega$ is a countably $1$-rectifiable set.

 Note by the way that even for a smooth  function on $\Omega$
its limits at the  points of the intervals $I_i$ from right and
left can  be different, so that
it is reasonable to introduce traces with two different values
in some points.

{\bf Notations.}
Denote by $A\Delta B$  the symmetric difference  $(A\setminus B)\cup(B\setminus A)$
 of
$A$ and $B$. $H_k$ denotes the $k$-dimensional Hausdorff measure and
$\V (A)$ denotes the  Lebesgue measure of  $A\subset\Rn$ or, equivalently,
its $n$-dimension Hausdorff measure.

The dimension $k=n-1$ will
play a special role for us and to be
short we denote $H_{n-1}=\mu$.
From here on words ``almost all'', ``measurable'', etc, will be used
with respect either to  $H_n$, or to  $H_{n-1}=\mu$, it will be clear
from the context to which one.

Denote by $B_p(r)$ the open ball of radius  $r$ centered at $p$
and by $\bar B_p(r)$ its closure.

$\Theta_A(p,k)$ denotes density with respect the measure $H_k$ of a set
$A$ at $p$; i.e.,
$$\Theta_A(p,k)= \lim_{r\to 0}v_k^{-1} r^{-k}H_k(A\cap B_p(r)),$$
where $v_k$ is the volume of the unit ball in $\mathbb{R}^k$. Note, that
we use  basically not densities, but  one-sided densities in the paper, see the next section.
\medskip

{\bf Countably rectifiable sets.}
There are several equivalent definitions of countably $(k, H_k)$-rectifiable sets.
One can find a detailed exposition in H.~Federer's monograph \cite{Fed},
Chapter 3, and more specifically 3.2.19, 3.2.25, 3.2.29.

The following definition is the most convenient for our purposes
\begin{definition}
The measurable set $A\subset\Rn$ is called countably  $(k,H_k)$-rectifiable if
there exists  a sequence of $C^1$-smooth $k$-dimensional surfaces $M_i$, $i=1,2,\dots$,
such that $A$ can be decomposed $A=\bigcup_{i=o}^\infty A_i$, where $\mu(A_0)=0$ and
$A_i\subset M_i$ for $i>0$. Moreover, the sets $A_i$ can be chosen such that
the following conditions hold:
\begin{equation}
\label{density_one}
\Theta_A(p,k)=0, \qquad
\Theta_{(A\setminus A_i)}(p,k)=1
\end{equation}
for almost all $p\in A_i$.
\end{definition}

We need the case $k=n-1$ only, so we call  countably
$(n-1,\mu)$-rectifiable set
{\em countably rectifiable} to be short.

Any countably rectifiable set $A$ has almost everywhere so called the approximative
tangent $(n-1)$-plane $T_pA$, which coincides with the tangent plane to
$M_i$ at$p$.
A point at which $T_pA$ exists and, in addition, equality \eqref{density_one} holds
is called the regular point. Thus, almost all (by measure $\mu$) points of $A$
are regular. We drop a definition of $T_pA$ because we need only the following its property:
for every sequence of positive numbers $r_j\to 0$, there exist positive numbers
 $\epsilon_j\to 0$ such that
\begin{equation}
\label{approx_plane}
\lim_{r_j\to 0}r^{1-n}\mu(B_p(r_j)\setminus L_{r_j\ep_j})=0,
\end{equation}
where $L_\delta$ is the $\delta$-neighborhood of  $T_pA$.
If $\nu$ is a normal to $T_pA$ at $p$ we will say that 
$\nu$ is a normal to $A$ at $p$.
\medskip

{\bf Functions.} As usually,
$BV(\Omega)$ means the class of locally summable  in  $\Omega$
functions such that their gradients are vector charges.
Denote by $\chi (E)$ the characteristic function of  $E$ and by
$P_\Omega(E)$ the perimeter of
$E\subset \Omega$; i.e.,  $P_\Omega(E)=\|\chi_E\|_{BV(\Omega)}$.
(We use notation
$\|f\|_{BV(\Omega)}=\text{var }\text{grad}f(\Omega)$.)
For more details see \cite{M}, \cite{BM}, \cite{Ziem}, \cite{Dzh}.

We will need the Fleming--Rishel formula
 \cite{FR}
\begin{equation}
\label{Flem-Rish_formula_0}
\|f\|_{BV(\Omega)}=\int_{-\infty}^{\infty}P_\Omega(E_t)\,dt,
\end{equation}
where $f\in BV(\Omega)$,
$E_t=\{x \mid f(x)>t\}$, and also the following formula closely connected with it
\begin{equation}
\label{Flem-Rish_formula}
\nabla f(E)=\int_{-\infty}^{+\infty}\nabla \chi_{E_t}(E)\, dt,
\end{equation}
where  $E$ is ant measurable subset of
 $\Omega$, see for instance Theorem 14 in \cite{BM} or Lemma  6.6.5/1 in \cite{M}.

 \begin{remark}
 \label{loc_perim}
 We will often consider sets $E$ for which $P_\Omega (E)<\infty$.
 For instance, it can be sets $E_t$ of points where a function $f$ greater than $t$.
 If considerations are local then the  finiteness perimeter condition can be
 replaced by the assumption that a set $E\cap\Omega$ has
 \emph{locally finite perimeter}; i.e.,  $P_{\Omega\cap Q} (E)<\infty$
 for any bounded region $Q$.
\end{remark}

\section{One-sided densities}
Let us consider a measurable set $E\subset\Rn$. Let  $\nu$ be a unit vector
at a point $x\in \Rn$.
Denote
$B^{\nu}_x(r)=B_x(r)\cap\{y\mid (y-x)\nu\geq 0\}$.
The limit  $$\Theta_E^{\nu}(x)=
\lim_{r\to 0}2v_n^{-1} r^{-n}H_n(B^{\nu}_x(r)\cap E).$$
is called
\emph{one-sided density} of the set $E$ at $x$ with respect to
$\nu$.

Upper and lower one-sided densities
$\overline {\Theta}_E^{\nu}(x)$,
$\underline{\Theta}_E^{\nu}(x)$
are defined analogically as upper and lower limits.
Now let $x$ be a regular point of the countably rectifiable set  $A$.
Then there are two normals to $A$ at $x$ and, correspondingly, it is naturally
to consider two one-sided densities with respect to  $A$, namely $\Theta_E^{\nu}(x)$ and
$\Theta_E^{-\nu}(x)$.

We often consider the boundary of  $\Omega$ in the capacity of $A$ assuming that
the boundary is a countably rectifiable set. In such cases we suppose usually that
$E\subset\Omega$.

\begin{remark}
\label{Rem_set_halfspace}
It is easy to see that if a set  $G$ is measurable and $\Theta^\nu_G(x)=1$ then
\begin{equation}
\Theta^\nu_E(x)=
\lim_{r\to 0}
\frac
{H_n(B_x^\nu(r)\cap G\cap E)}
{H_n(B_x^\nu(r)\cap G)}
=
\lim_{r\to 0}2v_n^{-1}r^{-n}{H_n(B_x^\nu(r)\cap G\cap E)}.
\end{equation}
\end{remark}

The following statement is a simple corollary of the isoperimetric inequality for
subsets of a ball.

\begin{lemma}
Let $E$ be an  measurable set with a finite perimeter,
$Q=\{x\in\Rn\mid \sum x_i^2<1,\,\, a<x_n<1\}$, where $a\leq 1/2$.
Then the following isoperimetric inequality holds
\begin{equation}
\label{isoper}
\min\{H_n(Q\cap E), \, H_n(Q\setminus E)\}
 \leq c_n P_Q(E)^{\frac{n}{n-1}},
 \end{equation}
 where  $c_n>0$ depends on dimension only.
\end{lemma}

\begin{lemma}
Let the boundary of a  region $\Omega$ is a countably rectifiable set. Then
either $\Theta_{\Omega}^{\nu}(x)=1$, or $\Theta_{\Rn\setminus\Omega}^{\nu}(x)=1$
at each regular point $x\in \D\Omega$ and for every normal $\nu(x)$ to $\D\Omega$.
\end{lemma}

Note that for normals
$\nu$, $-\nu$, any combination of values 0 and 1 for one-sided densities are possible.
That can happen even on a set of positive $\mu$-measure.

\begin{proof} Let  $\nu$ be a normal at a regular point $x\in \D\Omega$.
Consider semi-balls  $B^{\nu}_i=B^{\nu}_x(r_i)$, where
 $r_i\to +0$ as $i\to\infty$. Denote by  $C_ i$ intersection of
 $\epsilon_i r_i$-neighborhood of the plane
 $T_x$ with $B^{\nu}_i$, $A_i=B^{\nu}_i\setminus C_ i$.
It is clear that $\V(C_i)<v_{n-1}\ep_ir_i^n$.
By   \eqref{approx_plane} the inequalities
$P_{A_i}\leq\mu (A_i\cap\D\Omega)< \epsilon r_i^{n-1}$ hold for large  $i$ and sufficiently
small $\ep_i$.
Now the lemma follows immediately from the isoperimetric inequality
 \eqref{isoper} applied to the region
$A_i$ and the set $A_i\cap \Omega$.
\end{proof}

\begin{example}
Consider a sequence of small bubbles (disjoint round balls) $B_{x_i}(r_i)$
located in the unit open ball $B_0(1)$. It is easy to choose
these bubbles in such a way that all the points  $p\in S_0(1)$
are the limits of some subsequences of the bubbles and,
 besides, there is no other limit points. In addition,  suppose that the
radii of these balls vanish so fast that $\sum_i r_i^{n-1}<\infty$.
Define $\Omega=\bigcup B_{x_i}(r_i)$. Its boundary is rectifiable.
This set is not connected but in dimensions
$n>2$, one can connect the bubbles by very thin tubules such that
the new set $\Omega$
(completed with bubbles) becomes a region with rectifiable boundary.
The sphere
$S_0(1)$ belongs to the boundary of  $\Omega$. So almost all the points of this sphere
are regular points of $\D\Omega$. However they do not belong to the reduced
boundary of
$\Omega$; i.e., the set
$S_0(1)\bigcap \D^*\Omega$ is empty.   Moreover,  bubbles can be
chosen in such a way that at every point $x$
of the sphere $S_0(1)$,
 the condition  $\Theta_{\Omega}^{\nu}(x)=0$ holds for every normal.
\end{example}

Denote by $\Gamma$ the set of all points $x\in\D\Omega$ such
that $\Theta^\nu_\Omega (x)=1$ for at least one normal $\nu$.
It is not difficult to see that  $\D^*\Omega\subset \Gamma$.
Indeed, the vector $\nu_F$ is the normal in the
sense of Federer if and only if
$\Theta_\Omega^{-\nu_F}(x)=1$ and
$\Theta_\Omega^{\nu_F}(x)=0$.

\begin{remark}
\label{finit_per_rect}
It is well known that $P(\Omega)=\Hn(\D^*\Omega)$.
Recall that if \mbox{$P(\Omega)<\infty$,} then
$\text{var}\nabla \chi_\Omega (\D\Omega\setminus\D^*\Omega)=0$ and
\begin{equation}
\label{integr_normal}
\nabla\chi_\Omega (E)=-\int_E \nu_F(x)\,\mu(dx)
\end{equation}
for any measurable set $E\subset\D^*\Omega$, see for instance  \cite{BM},
Theorem  6.2.2/1.
\end{remark}

\begin{lemma}
\label{normal_field}
Any countably rectifiable set $A$ can be equipped with а measurable field
 $\nu$ of (unit) normals.
\end{lemma}

\begin{proof} The set  $A$, up to a subset of measure 0,
is located on  $(n-1)$-dimensional $C^1$-smooth
manifolds $M_i$ of some countable family.
It is not difficult to see that almost each point
$x\in A$ belongs to only one surface  $M_i$.
Let us orient every manifold $M_i$ by a continuous field of normals.
Since the approximative tangent plane to $A$ at $x$ coincides with
the tangent plane $T_xM_i$ and the intersection $A\cap M_i$ is measurable,
we obtain a measurable field of normals to $A$ by choosing normals $\nu(x)$ to $M_i$
in the capacity of normals to $A$.
\end{proof}

\begin{remark}
\label{def_stand_nor_field}
It is clear that a  measurable  vector field of unit normals is not unique,
there are infinitely many of such vector fields.
Let us fix some vector field $\nu$ constructed in Lemma \ref{normal_field}.
 It is not only measurable but is located on
$C^1$-smooth surfaces $M_i$ from a chosen  family and continuous along every such
surface.
 Besides, if a countably rectifiable set $A$ is the boundary
of a region $\Omega$, $A=\D\Omega$, then the vector field $\nu$ can be chosen so that,
at points $x\in\D^*\Omega$,
vectors $\nu(x)$ is directed opposite to normals in the sense of Federer.
A vector field   having such properties is called
\emph{standard}.
\end{remark}

\begin{lemma}
\label{measurable_sec}
Let $A$ be a
countably rectifiable set,
$\nu$ be a measurable field of normals to
$A$, and $E$ be a
measurable subset of $\Rn$.Then the sets
$\{x\in A\mid \Theta_E^\nu (x)=1\}$ and
$\{x\in A\mid \Theta_E^\nu (x)=0\}$
are
measurable.
\end{lemma}
\begin{proof}
First assume that  vector field $\nu$ is standard and a family of surfaces
 $\{M_i\}$ is chosen as above, in Remark  \ref{def_stand_nor_field}.
The sets $M_i\cap A$ are
measurable.
The functions
$\phi_i^r(x)=2v_n^{-1} r^{-n}H_n(B^\nu_x(r)\cap E)$ defined on  $M_i\cap A$
are continuous. In particular they are
measurable. Let us extend these functions to all $A$ by zero.
Their sum $\phi^r=\sum_i\phi_i^r$ defined on   $A$ is measurable too.
Therefore,
the  functions
$\underline{\phi}(x)=\liminf_{r\to 0}\phi^r(x)$ and
$\overline{\phi}(x)=\limsup_{r\to 0}\phi^r(x)$ are measurable and hence
the sets
$$\{x\in A \mid  \Theta_E^\nu (x)=0\}
=
\{x\in A \mid  \underline{\phi}(x)=0\},$$
$$\{x\in A \mid  \Theta_E^\nu (x)=1\}
=
\{x\in A \mid  \overline{\phi}(x)=1\}$$
are  measurable.
The same holds for the field $-\nu$ as well.
Now let $\tilde\nu$ be any measurable unit vector field of normals to
$\D\Omega$. Then the sets  $\{\nu=\tilde\nu\}$ and
 $\{-\nu=\tilde\nu\}$ are measurable, and thereby the set
 $\{x\in A\mid\Theta_E^{\tilde\nu} (x)=0\}$ and
$\{x\in A\mid\Theta_E^{\tilde\nu} (x)=1\}$
are measurable too.
\end{proof}

Let a set $A$ be countably rectifiable, $P(E)<\infty$,
and  $\nu$  be a normal to $A$ at
$x$. Denote
$$\D_A^\nu E=
\{x\in A\mid
\Theta_{E}^{\nu}(x)=1\},
$$
\begin{equation}
\D^1_AE=(\D_A^\nu E)\cup(\D_A^{-\nu} E),\quad\quad
\D^2_AE=(\D_A^\nu E)\cap(\D_A^{-\nu} E).
\end{equation}

Roughly speaking, $\D^1_AE$ is the set of points of $A$ such that $E$
``adjoins'' to $A$ with one-sided density  1 at least from one side and
$\D^2_AE$ is the part of  $A$ such that
$E$ ``adjoins''  with one-sided density  1 from both  sides.

Note that the following formulas hold:
\begin{equation}
\Gamma=\D_{\D\Omega}^1 \Omega,\quad\quad
\D^\nu_{\D\Omega}E=\D^\nu_\Gamma E.
\end{equation}

\medskip

We will use Lemma 6.6.3/1 from \cite{M} (or, that is the same, Lemma 13 from \cite{BM}).
The lemma is about the trace of a characteristic function. As
the notion of trace be
introduced later,  we formulate the lemma in a convenient form.
\begin{lemma}
\label{trace_charact}
Let $P(\Omega)<\infty$,
$E\subset \Omega$, $P_\Omega(E)<\infty$.
Then for almost all $x\in \D^*\Omega$
\begin{equation}
\label{equa_trace_charact}
\chi_{\D^*E}(x)=\lim_{r\to 0}
\frac{
\int_{B_x(r)}\chi_E\,dx}
{\V(B_x(r)\cap \Omega)}
=
\lim_{r\to 0}\frac{\V(B_x(r)\cap E)}{\V(B_x(r)\cap \Omega)}.
\end{equation}
\end{lemma}

For sure, only the first equality is essential, while  the latter one  is trivial.

\begin{remark}
\label{Rem_4}
In Lemma \ref{trace_charact}, the condition $E\subset\Omega$
can be dropped if one replaces
 $E$  to $E\cap\Omega$ and the condition
 $P_\Omega(E)<\infty$    to $P(E)<\infty$.
\end{remark}

The following lemma is the key one for our subsequent considerations.

\begin{lemma}
\label{o_one}
Let $A$ be a countably rectifiable set,  $\nu$ be a measurable
field of normals along
$A$, and  $P(E)<\infty$.
Then $\mu$-almost everywhere on $A$, one-sided densities
$\Theta^\nu_E(x)$  equal either 0, or 1.
\end{lemma}

\begin{proof}
It suffices to prove the lemma for  standard normal vector fields
and taking into account only  regular points of $A$
(see Lemma \ref{normal_field} and Remark \ref{def_stand_nor_field}).

1. First let  $A$ be $C^1$-smooth $(n-1)$-dimensional manifold $M$.
Since our statement is local, we can suppose that
$M$ divides some its neighborhood bounded by a smooth hypersurface onto
two semi-neighborhoods,  $\Omega_1$ and
 $\Omega_2$. Set $E_i=\Omega_i\cap E$, $i=1,2$. It is clear that $P(E_i)<\infty$.

Note that $\chi_{\D^*E_1}(x)$ equals 1 if  $x\in \D^*E_1\cap M$ and equals  0
if $x\in M\setminus \D^*E_1$.
Therefore, applying Lemma \ref{trace_charact} to the sets $E=E_1$ and
$\Omega=\Omega_1$ and Remark \ref{Rem_set_halfspace} for $G=\Omega_1$,
we see that for almost all points
$x\in M$ the one-sided density
$\Theta^\nu_{E_1}(x)$ is equal either 0 or 1, where $\nu$ is the  normal к $M$
directed to the side of $\Omega_1$. The same is true for $E_2$ and $\Omega_2$.
Finally, since
$$
1\geq \Theta^\nu_E(x)=\Theta^\nu_{E_1}(x)+\Theta^\nu_{E_2}(x),
$$
we see that the lemma is proved for $A=M$.

2. Let us pass to the general case.
 Let $\{M_i\}$ be a family of $C^1$-smooth submanifolds, mentioned in the definition
 of  standard normal fields.
In the item 1, the lemma was already proved  for each $M_i$.
The intersection $A\cap M_i$ is   $\mu$-measurable, and one-sided density at a point
depends on $\nu$ and $E$ only. Thus $\Theta_{E}^{\nu}(x)$ is equal either 0 or 1
almost everywhere on $A\cap M_i$. Since $A$, up to a set of measure 0, coincides
with the  union of sets $A\cap M_i$,
the lemma is proved.
\end{proof}

\begin{corollary}
\label{Omega_0_1}
Let $\Omega$ be a region such that its boundary is  a countably rectifiable set.
If $E\subset \Omega$ and $P(E)<\infty$,   then for any (measurable) field
$\nu$ of normals to  $\D\Omega$, one-sided densities $\Theta_E^\nu$
are equal almost everywhere either 0 or 1.
\end{corollary}
It is clear now, that, for the reduced boundary of any set $E$ with $P(E)<\infty$,
the following holds:
\begin{equation}
\label{Dstar}
A \cap\D^*E=(\D_A^1 E)\setminus(\D_A^2 E),
\end{equation}
in particular
\begin{equation}
\label{DstarOmega}
\D^*\Omega=(\D_\Gamma^1 \Omega)\setminus(\D_\Gamma^2 \Omega).
\end{equation}

\begin{corollary}
\label{cor_generalization} Let $x$ be a regular point of $\D\Omega$.
Suppose that
$\Theta_{G_1}^\nu(x)=\Theta_{G_2}^\nu(x)=1$ for some sets $G_1$, $G_2$.
In addition, assume that there is a family of sets
 $\mathcal{B}^\nu_x(r)$ such that
\begin{equation}
B_x(\rho_1(r))
\cap
G_1
\subset
\mathcal{B}^\nu_x(r)
\subset
B_x(\rho_2(r))
\cap
G_2
,
\end{equation}
where $\rho_2(r)\to0$ as $r\to0$.
Then the equality
\begin{equation}
\Theta^\nu_E(x)=
\lim_{r\to0}
\frac
{H_n(\mathcal{B}^\nu_x(r)\cap E)}
{H_n(\mathcal{B}^\nu_x(r))}
\end{equation}
holds for any set $E\subset\Rn$ with finite perimeter.
\end{corollary}

This corollary allows to consider one-sided densities
for sets with finite perimeters in any  $C^1$-smooth manifold with a continuous
metric tensor. Therefore further considerations  are applicable not only to
$\mathbb{R}^n$, but also to any such a manifold.
\section{Trace on a countably rectifiable set}
Here we define trace on a countably rectifiable set for a function defined in $\Omega$.
Within this section we do not require
function to belong to $BV(\Omega)$.
Instead of that we only suppose that  the sets
$E_t=\{x\in\Omega \mid f(x)>t\}$ have finite perimeters for almost all $t$.
We call functions $BV$-{\em similar} if they have such property.
(As it was mentioned in Remark \ref{loc_perim}, it would be sufficiently to suppose
that $E_t$ has locally finite perimeter.)

Let  a countably rectifiable set $A$ is contained in the closure $\bar\Omega$ of a
region $\Omega$.
Let us define trace\footnote{Our terminology is different of one in
\cite{BM}, \cite{M}. Namely, we use terms  trace and average  trace
instead of rough trace and trace.} $f^{\nu}(x)$ with respect to normal
$\nu$ at $x\in\D_A^\nu\Omega$
for a  $BV$-{\em similar} function $f$ as follows:
$$f^{\nu}(x)=\sup\{t\mid x\in\D_A^{\nu}E_t\}.$$
We can suppose  (this change nothing), that supremum is taken only over
$t$ such that  $P(E_t)<\infty$.
Moreover we assume that $\sup\emptyset=-\infty$.

Let us emphasize, that trace is defined not everywhere on $A$.
However if one extend $f$ to all $\Rn$ (for instance, by a constant),
so that $A=\D_A^\nu(\Rn\setminus A)$, then $f^{\nu}$ is defined
on $A$ everywhere.

In the case
$x\in\D_A^2 \Omega $ we also define the upper and lover traces by equations
$$f^*(x)=\max\{f^{\nu}(x),f^{-\nu}(x)\},\quad
f_*(x)=\min\{f^{\nu}(x),f^{-\nu}(x)\}.$$
If $x\in A \cap\D^*\Omega=(\D_A^1 \Omega)\setminus(\D_A^2 \Omega)$, we put
$f^*(x)=f^\nu (x)$,
where $-\nu$ is the normal in the sense of Federer. In this case we do not define
 $f_*(x)$ at all. However,  if  $f$ is extended on all $\Rn$
 (for instance, by a constant) then $A=\D_A^2(\Rn\setminus A)=\D_A^1(\Rn\setminus A)$ and the upper and lower traces
are defined on all $A$.

It is clear, that
$f^*(x)=\sup\{t\mid x\in\capU E_t\}$,\,
$f_*(x)=\sup\{t\mid x\in\capD E_t\}$.

\begin{lemma}
\label{Main_trace}
Let $A\subset\bar\Omega$ be a  countably rectifiable set,
$\nu$ be a  measurable field of normals to $A$.
Then for any  $BV$-similar function $f$ its trace $f^{\nu}$
on $\D^\nu_A\Omega$ is measurable and

\begin{equation}
\label{Main_trace_form3}
\Hn(\{x\in \D^\nu_A\Omega \mid  f^{\nu}(x)\geq t\})=\Hn(\D_A^\nu{E_t})
\end{equation}
for almost all $t\in\mathbb R$.
\end{lemma}

\begin{remark}
1) Analogously to Lemma  \ref{Main_trace}, it can be proved that traces
$f^*$ and $f_*$ are
measurable as well and
\begin{align}
\label{Main_trace_form1}
\Hn(\{x\in \capU\Omega  \mid  \trU{f}(x)\geq t\})&=\Hn(\capU {E_t}),\\
\label{Main_trace_form2}
\Hn(\{x\in \capD\Omega \mid  \trD{f}(x)\geq t\})&=\Hn(\capD{E_t}).
\end{align}

2) In fact, instead of \eqref{Main_trace_form3}, we will prove, that
$$\Hn\Big(\{x\in \D_A^\nu\Omega\mid\trN{f}{\nu}(x)\geq t\}
\Delta\D_A^\nu{E_t}\Big)=0$$
for all  $t$  except a countable subset.

3) Note that in \eqref{Main_trace_form3}--
\eqref{Main_trace_form2} unstrict inequalities can be replaced by strict ones.
\end{remark}

\begin{proof}
Denote
$B_t=\{x\in \D_A^\nu \Omega \mid  \trN{f}{\nu}(x)\geq t\}$,
$Y_t=\D_A^{\nu}{E_t}$ and
$X_t=B_t\setminus Y_t$.
It is easy to see that
$B_t\supset Y_t$. Thus, it remains to prove that
$\mu(X_t)=0$.

The sets $Y_t$ are measurable, and the sets
$X_t$ are disjoint. It is not difficult to see that
the inclusions
$Y_{t_0}\supset Y_{t_1}$ and
$Y_{t_0}\cup X_{t_0}\supset Y_{t_1}\cup X_{t_1}$ hold for $t_0<t_1$.
The latter inclusion implies that
$Y_{t_0}\supset X_{t_1}$. So
$$(\bigcap_{t<t_1}Y_t)\setminus Y_{t_1}\supset X_{t_1}.$$
From the other hand the sets $(\bigcap_{t<t_1}Y_t)\setminus Y_{t_1}$ are
measurable and disjoint. Therefore
$\mu\big(\big(\bigcap_{t<t_1}Y_t\big)\setminus Y_{t_1}\big)=0$ for almost all
$t_1\in\mathbb{R}$. From this it follows that the sets $X_t$ are subsets of
measure zero sets for almost all $t\in\mathbb{R}$. In particular, they are measurable.
It follows  that the sets
$B_t$ are measurable.
\end{proof}

\begin{lemma}
\label{minus_minus}
Let $A\subset\bar\Omega$ be a  countably rectifiable set,   $f$ be a $BV$-similar function.
Then the inequality
\begin{equation}
-\trN{f}{\nu}(x)=\trN{(-f)}{\nu}(x)
\end{equation}
holds for almost all $x\in \D_A^\nu\Omega$.
\end{lemma}

\begin{proof}
Lemma \ref{minus_minus} is equivalent to the statement that the equality
$$
\sup\{t\mid x\in\D_A^{\nu}{E_t}\}=
\inf\{t\mid x\in\D_A^{\nu}(\Omega\setminus E_t)\}.
$$
holds for almost all $x\in A$.
The last equality means that
$$
\sup\{t\mid \underline{\Theta}_{E_t}^{\nu}(x)=1\}=
\inf\{t\mid \underline{\Theta}_{(\Omega\setminus E_t)}^{\nu}(x)=1\}.
$$
In its turn, this is equivalent to the equality
$$
\sup\{t\mid \underline{\Theta}_{E_t}^{\nu}(x)=1\}=
\inf\{t\mid \overline{\Theta}_{E_t}^{\nu}(x)=0\}.
$$
Denote by $L$ and $R$ the left and the right parts of the last equality.
It is not difficult to see that the  functions $\overline{\Theta}_{E_t}^{\nu}(x)$ and
$\underline{\Theta}_{E_t}^{\nu}(x)$ are not increasing in $t$.
Therefore $L\leq R$. Consider the set of the points
$x$ such that $L(x)<R(x)$. It suffices to prove that  $\mu$-measure of
this set equals zero.

For this let us choose a countable everywhere dense set $\{t_i\}_{i=1}^{\infty}$
such that
$P(E_{t_i})<\infty$. If $L(x)<R(x)$ then there exists
$t_i$ such that  $L(x)<t_i<R(x)$. Now our assertion follows from Lemma
\ref{o_one} applied to the set
$E_{t_i}$.

\end{proof}

\begin{corollary}
\label{cor_pos_neg}
For any $BV$-similar function $f$ and for almost all
$x\in A$ the following equalities hold:
\begin{equation}
\label{pos_neg_part}
(\trN{f}{\nu})^+=\trN{(f^+)}{\nu}, \quad \quad
(\trN{f}{\nu})^-=\trN{(f^-)}{\nu}.
\end{equation}
\end{corollary}
\begin{proof}
The first equality can be derived directly from definitions. The letter one easily
follows from Lemma \ref{minus_minus}. Indeed,

$(f^-)^\nu=((-f)^+)^\nu=((-f)^\nu)^+=(-(f^\nu))^+=(f^\nu)^-$.

\end{proof}

\begin{lemma}
\label{lem_sum_trace}
For any $BV$-similar functions $f, g$ and almost all
$x\in A$ the following equality holds:
\begin{equation}
\label{sum_trace}
\trN{(f+g)}{\nu}(x)=\trN{f}{\nu}(x)+\trN{g}{\nu}(x).
\end{equation}
\end{lemma}

\begin{proof} First prove that
$\trN{(f+g)}{\nu}(x)\geq \trN{f}{\nu}(x)+\trN{g}{\nu}(x)$ for all $x\in\Gamma$.
Indeed, choose numbers $F<\trN{f}{\nu}(x)$ and $G<\trN{g}{\nu}(x)$ such that
the sets  $E_{F}^{f}=\{x\mid f(x)>F\}$ and $E_{G}^{g}=\{x\mid g(x)>G\}$ have
finite perimeters. Then $\Theta^\nu_{E_{F}^{f}}(x)=1$ and  $\Theta^\nu_{E_{G}^{g}}(x)=1$.

Denote $W=E_{F+G}^{f+g}$. We have
$$W=\{x\mid f(x)+g(x)>F+G\}\supset E_{F}^{f}\cap E_{G}^{g}.$$
Therefore $\Theta^\nu_W(x)=1$ and so
$$\trN{(f+g)}{\nu}(x)=\sup\{t\mid \Theta_{E_t^{f+g}}^\nu=1\}\geq F+G$$
Passing to the limits as $F\rightarrow\trN{f}{\nu}(x)$ and $G\rightarrow\trN{g}{\nu}(x)$,
we get
$$\trN{(f+g)}{\nu}(x)\geq \trN{f}{\nu}(x)+\trN{g}{\nu}(x).$$

Now we will derive the opposite inequality using Lemma \ref{minus_minus}.Indeed
for almost all $x\in A$ we have:
$$
-\trN{(f+g)}{\nu}(x)=\trN{((-f)+(-g))}{\nu}(x)\geq
\trN{(-f)}{\nu}(x)+\trN{(-g)}{\nu}(x)=-\trN{f}{\nu}(x)-\trN{g}{\nu}(x).
$$
\end{proof}

\begin{lemma}
\label{lem_composition}
Let the function $\phi:\R\to\R$ be increasing and left-continuous.
If  functions $f$ and $\phi\circ f$ are $BV$-similar then
\begin{equation}
\trN{(\phi\circ f)}{\nu}(x)=\phi(\trN{f}{\nu}(x))
\end{equation}
for almost all $x\in A$.
\end{lemma}

\begin{proof}
The lemma easily follows from the equality
\begin{equation*}
\{x\in\Omega \mid  (\phi\circ f(x)))\geq \phi(t)\}=
\{x\in\Omega \mid  f(x)\geq t\}.
\end{equation*}

\end{proof}

\begin{remark}
\label{Rem_BV-sim_diff}
1) Suppose that Hausdorff measure $H_{1}(\phi^{-1}(E))=0$ for any set $E$ of measure 0.
Then the statement that  $\phi\circ f$ is $BV$-similar implies that the function $f$
is $BV$-similar.
This assertion holds definitely if (locally)\linebreak
$|\phi(x)-\phi(y)|\geq \text{const}|x-y|$. The last condition obviously holds if
 $\phi\in C^1$  and $\phi'\neq0$.

2) In the lemma  increasing of $\phi$ can be replaced by the assumption
that the set $\phi^{-1}((t,+\infty))$ is a finite union of intervals and rays for almost all
$t$.
\end{remark}

\begin{lemma} If functions
$f$, $g$, and $fg$ are $BV$-similar then
\begin{equation}
\label{prod_trace}
\trN{(fg)}{\nu}(x)=\trN{f}{\nu}(x)\trN{g}{\nu}(x).
\end{equation}
for almost all
$x\in\Gamma$.
\end{lemma}
\begin{proof}
It is enough to prove   \eqref{prod_trace} only for
$f,g\geq1$.
It follows from Lemma \ref{minus_minus}, Corollary \ref{cor_pos_neg}
and the equality
$f=(f^++1)-(f^-+1)$.

In this case  Lemma
\ref{lem_sum_trace}, Lemma \ref{lem_composition},
and Remark \ref{Rem_BV-sim_diff}  imply
\begin{equation*}
\trN{(fg)}{\nu}=\trN{(e^{\ln(fg)})}{\nu}=
e^{\trN{(\ln f+\ln g)}{\nu}}=
e^{\trN{(\ln f)}{\nu}+\trN{(\ln g)}{\nu}}=
e^{\ln (\trN{f}{\nu})+\ln (\trN{g}{\nu})}=
\trN{f}{\nu} \trN{g}{\nu}.
\end{equation*}
\end{proof}

\section{Integral formula for norm of trace}
\begin{definition}
Let us define a norm of the trace on $\D\Omega$ of a function $f\in BV(\Omega)$
as follows:
\begin{equation}
\label{tracenorm}
\tracenorm{f}
=
\int_{\D^*\Omega} |\trU{f}|\,d\Hn
+\int_{\capD \Omega} (\trU{f}- \trD{f})\,d\Hn.
\end{equation}
If $\tracenorm{f}<\infty$ we will say that  $f$ has the summable trace.
\end{definition}

\begin{lemma}
\label{Lem_trace_norm}
\begin{equation}
\label{form_norm_plus_minus}
\tracenorm{f}
=\tracenorm{f^+}
+\tracenorm{f^-}.
\end{equation}
\end{lemma}
\begin{proof}
We have
\begin{equation*}
\begin{split}
\trU{f}-\trD{f}=&|\trN{f}{\nu}-\trN{f}{-\nu}|=
|\trN{((f)^+)}{\nu}-\trN{((f)^-)}{\nu}-\trN{((f)^+)}{-\nu}
+\trN{((f)^-)}{-\nu}|\\
=&
|\trN{(f^+)}{\nu}-\trN{(f^+)}{-\nu}|+|\trN{(f^-)}{-\nu}-\trN{(f^-)}{-\nu}|\\
=&
(\trU{(f^+)}-\trD{(f^+)})+(\trU{(f^-)}-\trD{(f^-)}).
\end{split}
\end{equation*}
\end{proof}
\begin{lemma}
Suppose that a function $f\in\BVO$ is nonnegative   and has the
summable trace on $\D\Omega$.
Moreover, let a vector-function $\eta\colon\Gamma\to\R^k$, $k\geq 1$,
is measurable and bounded.
Then
\begin{equation}
\label{one_int_form}
\int_0^{+\infty}
\int_{\Gamma\cap\D^*E_t}\eta\,d\Hn\,dt=
\int_{\D^*\Omega}\trU{f}\eta\,d\Hn+
\int_{\capD\Omega}(\trU{f}-\trD{f})\eta\,d\Hn.
\end{equation}
\end{lemma}

\begin{proof}
Clearly,  it  suffices to consider only the  case  $k=1$.
Define
\begin{gather*}
\{x\in\capU{\Omega}\mid \trU{f}>t\}=E_t^1,\;\;
\{x\in\capD{\Omega}\mid \trU{f}>t\}=E_t^2 \\
\{x\in\capD{\Omega}\mid \trD{f}>t\}=L_t^2, \,\,
\{x\in\D^*{\Omega}\mid \trU{f}>t\}=E_t^*.
\end{gather*}
By \eqref{Dstar} and Lemma
\ref{Main_trace},  we have
\begin{gather*}
\label{proof_one_int_form}
\int^{+\infty}_0
\int_{\Gamma\cap\D^*E_t}\eta\,d\Hn\, dt
=\int^{+\infty}_0
\Big(
\int_{\Gamma\cap\capU{E_t}}\eta\,d\Hn-
\int_{\Gamma\cap\capD{E_t}}\eta\,d\Hn
\Big)
\, dt\\
=
\int^{+\infty}_0
\Big(
\int_{E_t^1}\eta\,d\Hn
-
\int_{L_t^2}\eta\,d\Hn
\Big)
\, dt\\
=
\int^{+\infty}_0
\Big(
\int_{E_t^*}\eta\,d\Hn
+
\int_{E_t^2}\eta\,d\Hn
-
\int_{L_t^2}\eta\,d\Hn
\Big)
\, dt\\
=
\int^{+\infty}_0
\Big(
\int_{E_t^*}\eta\,d\Hn
+
\int_{E_t^2}\eta\,d\Hn-\int_{L_t^2}\eta\,d\Hn
\Big)
\, dt\\
= \int_{\D^*\Omega} \trU{f}\eta\,d\Hn
+\int_{\capD \Omega} (\trU{f}-\trD{f})\eta\,d\Hn.
\end{gather*}

\end{proof}

\begin{corollary}
If a function $f\in\BVO$ is nonnegative then
\begin{equation}
\label{one_int_form_eta=1}
\tracenorm{f}=
\int_0^{+\infty}
\Hn(\Gamma\cap\D^*E_t)\,dt.
\end{equation}
In addition,  $f$ has the summable trace if and only if the right part
of
\eqref{one_int_form_eta=1}
is finite.
\end{corollary}

Indeed, if $\tracenorm{f}<\infty$ then we can obtain \eqref{one_int_form_eta=1}
substituting $\eta=1$ in \eqref{one_int_form}.
Now let  the right part of \eqref{one_int_form_eta=1} is finite.
Then it suffices to substitute $\eta=1$
in the latter equalities of the proof of Lemma \ref{Lem_trace_norm} and read them
from right to left to prove \eqref{one_int_form_eta=1}.

\section{Summability of  traces and 
integral inequalities}

In this and the next sections, we are going to show that in fact  all the  integral
inequalities and other results on traces obtained in \cite{BM}, \cite{M}
can be generalized to the case when the boundary of a region
is a countably rectifiable set.
As the integral inequalities obtained in \cite{M} are various, we restrict
ourselves with only  key examples.

For a set $A\subset \bar\Omega$, denote by $\tau_A$ the infimum of numbers $\beta$
such that the inequality
$\mu(\D^*E\cap\Gamma)\leq\beta\mu(\D^*E\cap\Omega)$ holds for all
$E\subset \Omega$  satisfying
$$\V(A\cap E)+\mu(A\cap\D^*E)=0.$$

Note that $\tau_A$ goes to infinity as
 $A$ vanishes.
Indeed, we can set $E=\Omega\setminus A$.

The following theorem generalizes Theorem 6.5.3/1 in \cite{M}.
 \begin{theorem}
 \label{summability}
Let the boundary  $\D\Omega$ of a region $\Omega$ be a countably rectifiable set,
 $D$ be a subset of $\bar\Omega$.
 Then for any function  $f\in BV(\Omega)$  satisfying the condition
$f(A\cap\Omega)=0$, $f^*(A\cap\Gamma)=0$ the  inequality
\begin{equation}
\label{norms_ineq}
\tracenorm{f}\leq\tau_A\|  f\|_{BV(\Omega)}
\end{equation}
 holds and the constant  $\tau_A$ is exact.
 \end{theorem}

 \begin{proof} We can assume that $\|  f\|_{BV(\Omega)}<\infty$.
Suppose for a while that $f\geq 0$. Note that
 $\V (A\cap E_t)+\mu (A\cap\D^*E_t)=0$ for almost all $t>0$. Then by Corollary
 \eqref{one_int_form_eta=1} and the definition of  $\tau_A$ we have
 \begin{equation}
\tracenorm{f}= \int_0^{+\infty}\Hn(\Gamma\cap\D^*E_t)\,dt\leq\tau_A
\int_0^{+\infty}P_{\Omega}(E_t)\,dt=\tau_A\|  f\|_{BV(\Omega)}.
\end{equation}

Let now the  function $f$ be not necessary nonnegative. By Lemma
\ref{form_norm_plus_minus} we have
\begin{equation}
\begin{split}
\tracenorm{f}&=\tracenorm{f^+}+\tracenorm{f^-}\\
&\leq
\tau_A(\|  f^+\|_{BV(\Omega)}+\|  f^-\|_{BV(\Omega)})=
\tau_A\|  f\|_{BV(\Omega)}.
\end{split}
\end{equation}
 \end{proof}

The next theorem generalizes Theorem 6.5.4/1 in \cite{M}.

\begin{theorem}
\label{local_emb}
Suppose that the  boundary of a region  $\Omega$ is a
countably rectifiable set.
Then in order for any
function $f \in BV(\Omega)$ to satisfy the inequality
\begin{equation}
\label{local_ineq}
\tracenorm{f} \leq k (\|f\|_{BV(\Omega)}+\|f\|_{L(\Omega)})
\end{equation}
with a constant $k$ independent on  $f$, it is necessary and sufficient
that there exists a constant  $\delta>0$ such that
the inequality
\begin{equation}\label{local_isop}
\mu(\D^*E\cap\D^*\Omega) \le k_1 P_\Omega(E)
\end{equation}
holds for every measurable  set
$E \subset \Omega$ with $\diam E\leq\delta$,
where the constant  $k_1$ does not depend on $E$.
\end{theorem}

To prove the necessity of  \eqref{local_ineq} it  suffices to insert
$f = \chi_E$ in \eqref{local_ineq}. The sufficiency can be derive from Theorem
\ref{summability} with the help of a partition of unity.
\medskip

Theorem 4 in \cite{BM} (or, that is the same, Theorem   6.5.2(1) in \cite{M})
can be naturally generalized to the case of  regions with countably rectifiable boundary
in the following form.
\begin{theorem}
\label{isoper_thm_2}
Let the boundary of a region $\Omega$ be a countably rectifiable set.
Then the inequality
\begin{equation}
\label{isoperOmega}
\inf_{c}
\left\{
\tracenorm{f-c}
\right\}
\leq
k\|f\|_{\BVO}
\end{equation}
is satisfied
with a constant  $k$ independent on  $f\in \BVO$
if and only if  the inequality
\begin{equation}
\label{summarized}
\min\big\{
\Hn(
\Gamma\cap\D^* E
),\,\,
\Hn(
\Gamma\cap\D^* (\Omega\setminus E)\}
\leq
k P_\Omega(E)
\end{equation}
holds for each set  $E\subset\Omega$ having the finite perimeter.
\end{theorem}

\begin{proof}
First note (cf. \eqref{Dstar}) that
\begin{align}
\label{5.3}
\Hn(\Gamma\cap\D^* E)&=
\Hn(\D^*\Omega\cap\capU E)+
\Hn(\capD\Omega\cap\D^* E),\\
\label{5.4}
\Hn(\Gamma\cap\D^*(\Omega\setminus E))&=
\Hn(\D^*\Omega\setminus\capU E)+
\Hn(\capD\Omega\cap\D^* E).
\end{align}
\textbf{Necessity.}
Let $E\subset\Omega$, $P_\Omega(E)<\infty$.
For the characteristic function $\chi_E$ of the set
 $E$ we have
\begin{align*}
k &P_\Omega(E)=
k\|\chi_E\|_{\BVO}\\
&\geq
\inf_{c}\big\{
\int_{\D^*\Omega}|\trU{(\chi_E)}(x)-c|\,d\Hn(x)
+
\int_{\capD\Omega}\big(\trU{(\chi_E)}(x)-\trD{(\chi_E)}(x)\big)\,d\Hn(x)
\big\}\\
&=\min_c \{|1-c|\Hn(\D^*\Omega\cap\capU E)
+|c|\Hn(\D^*\Omega\setminus\capU E)
+\Hn(\capD\Omega\cap\D^* E)\big\}\\
&=\min \big\{
\Hn(\D^*\Omega\cap\capU E),\,\,
\Hn(\D^*\Omega\setminus\capU E)
\big\})
+\Hn(\capD\Omega\cap\D^* E)\big\}.
\end{align*}
Jointly with
\eqref{5.3} and \eqref{5.4} this proves the inequality  \eqref{summarized}.

\textbf{Sufficiency.} If $\|f\|_{\BVO}<\infty$ then $P(E_t)<\infty$
for almost all $t$.
Taking into account \eqref{summarized}--\eqref{5.4}, by the Fleming--Rishel formula
\eqref{Flem-Rish_formula_0} we get
\begin{equation}
\begin{split}
\label{Only_if_1}
k&\|f\|_{BV(\Omega)}
=
k\int_{-\infty}^{+\infty}P_\Omega(E_t)\,dt\\
&\geq
\int_{-\infty}^{+\infty}
\Big(
\min \big\{
\Hn(\D^*\Omega\cap\capU E_t),
\Hn(\D^*\Omega\setminus\capU E_t)
\big\}
+\Hn(\capD\Omega\cap\D^* E_t)\Big)
\,dt.
\end{split}
\end{equation}

Denote
$t_0=\sup\big\{t\mid
\Hn(\D^*\Omega\cap\capU E_t)
\geq
\Hn(\D^*\Omega\setminus\capU E_t)
)\big\}$
and observe that
\mbox{$\Hn(\D^*\Omega\cap\capU E_t)$}
does not increase in $t$ and
$\Hn(\D^*\Omega\setminus\capU E_t)$ does not
decrease in $t$. Hence, by
\eqref{one_int_form_eta=1} we obtain
\begin{equation*}
\begin{split}
k\|f\|_{BV(\Omega)}
\geq&
\int_{t_0}^{+\infty}
\Hn(\Gamma\cap\D^* E_t)
\,dt
+\int_{-\infty}^{t_0}
\Hn(\Gamma\cap\D^*(\Omega\setminus E_t))
\,dt\\
=&
\tracenorm{(f-c)^+}+
\tracenorm{(f-c)^-}=
\tracenorm{f-c}.
\end{split}
\end{equation*}
So  \eqref{summarized} holds and the theorem is proved.
\end{proof}

\section{Extension of a function in $BV(\Omega)$ to all the space by a constant}

In this section we suppose everywhere that $P(\Omega)< \infty$ and $\D\Omega$
is a countably rectifiable set.

Let a function $f$ be defined in a region $\Omega\subset \Rn$. Denote by $f_c$ the
function
$f\colon\Rn\to \mathbb{R}$, defined by the condition
 $f_c(x)=f(x)$ for $x \in \Omega$ and $f_c(x)=c$ for $x \notin \Omega$,
where  $c$ is a constant.
\begin{lemma}
\label{lem-const-out}
The following equality
\begin{equation}
\begin{split}
\label{eq6.5.5/1}
\|f_c\|_{BV(\mathbb{R}^n)}=\|f\|_{BV(\Omega)}
+
\tracenorm{f-c}
\end{split}
\end{equation}
holds.
\end{lemma}

\begin{proof}
Without lost of generality we can assume $c=0$; indeed, it  is enough to consider
$f-c$ instead of $f$.
The equality
\eqref{form_norm_plus_minus}
allows to assume that  $f\geq0$.
As usual we set $E_t=\{x\in\Omega\mid f_0>t\}$.
Now by the equalities \eqref{Flem-Rish_formula_0} and  \eqref{one_int_form_eta=1} we have
\begin{equation*}
\label{eq6.5.5/2}
\begin{split}
\|f_0\|_{BV(\mathbb{R}^n)}&=
\int^{+\infty}_0 P(\{x\in\Rn\mid f_0>t\})\, dt\\
&=\int^{+\infty}_0
\Big(
P_{\Omega} (E_t)
+
\Hn(\Gamma\cap\D^*E_t)
\Big)
\, dt\\
&=\|f\|_{BV(\Omega)}
+
\tracenorm{f}.
\end{split}
\end{equation*}
\end{proof}

The question  can arise: if it is possible to enlarge $\Omega$ by removing $\D^2_\Gamma \Omega$
and thus to reduce our case to one when  normals in the sense of
Federer exist
almost everywhere on $\D\Omega$. Sometimes it is possible.
For instance, let
 $\Omega=D^2\setminus\cup_{i=1}^\infty I_i$ be
the disk with a sequence of intervals removed in such a way, that the sum of
lengths of $I_i$ is finite.
Then every $f\in\BVO$
such that
$$\int_{\cup_{i=1}^\infty I_i}(\trU{f}-\trD{f})<\infty$$
can be extended to
  a function $\tilde f\in BV(D^2)$.
 Unfortunately a slightly more complicated example shows that this is not
 necessary the case.

 \begin{example}
Denote by $K\subset[0,1]$ a Cantor set of positive length.
Define the region $\Omega$ as follows:
\begin{equation}
\Omega=
B_{(0,0)}(2)
\setminus
\{
(x,y)
\mid
x\in[0,1],\,\,
|y|\leq(\dist(x,K))^2
\}.
\end{equation}
It is not difficult to see that both of the one-sided densities equal one at all
points of the set
  $K\times\{0\}$ and  $\D^2_\Gamma\Omega$ is just the set of these points.
Nevertheless it is impossible to enlarge  $\Omega$  so that to include this set
in the region.
\end{example}

\section{Embedding theorems}

The following theorem is a direct generalization of Theorem
6.5.7/1 in \cite{M}.
\begin{theorem}
\label{embedding_thm}
Suppose that  $\D\Omega$ is a countably $\Hn$-rectifiable set.
Then for every  function $f\in \BVO$ the inequality
\begin{equation}
\label{embedding}
\left[\int_\Omega f^{\frac{n}{n-1}}\,dx\right]^{\frac{n-1}{n}}
\leq
n c_n^{-\frac{1}{n}}
\left\{
\|f\|_\BVO
+
\tracenorm{f}
\right\}
\end{equation}
holds and the constant
$n c_n^{-\frac{1}{n}}$ is exact.
\end{theorem}
\begin{proof} By Corollary \ref{cor_pos_neg} and Lemma \ref{lem_sum_trace}
we can suppose that
$f\ge0$.
Just as in Theorem 7 in \cite{BM}, we get
\begin{equation}
\left[\int_\Omega|f|^{\frac{n}{n-1}}\,dx\right]^{\frac{n-1}{n}}
\leq
\int_0^{+\infty} H_n(E_t)^\frac{n-1}{n}\,dt,
\end{equation}
where as usual $E_t=\{x\in\Omega \mid f(x)>t\}$.

It follows from the isoperimetric inequality that
\begin{equation}
\begin{split}
\label{intermed_eqv}
H_n( E_t)^\frac{n-1}{n}
\leq&
n c_n^{-\frac{1}{n}}P_{\Rn}(E_t)\\
=&
n c_n^{-\frac{1}{n}}\left[P_\Omega(E_t)+
\Hn(\Gamma\cap\D^* (E_t))\right].
\end{split}
\end{equation}

Now the equations
 \eqref{intermed_eqv} and \eqref{one_int_form_eta=1} imply
\begin{equation*}
\begin{split}
n^{-1}c_n^{\frac1n}\left[\int_\Omega|f|^{\frac{n}{n-1}}\,dx\right]^{\frac{n-1}{n}}
&\leq
\int_{-\infty}^{+\infty}P_\Omega(E_t)\,dt
+\int_0^{+\infty} \Hn(\Gamma\cap\D^* (E_t))\,dt
\\
&=\|f\|_\BVO
+
\tracenorm{f}.
\end{split}
\end{equation*}
\end{proof}

Note that the multiplicative inequality 6.5.6 in \cite{M}
can also be generalized to our case.


\section{The Gauss--Ostrogradskiy formula}
\begin{theorem}[The Gauss--Ostrogradskiy formula] Let the boundary of a
region $\Gamma$ is a countably
$\Hn$-rectifiable set. Assume that
$\D\Omega$ is equipped with a standard field $\nu$ of unit normals and
the trace of a function $f\in\BVO$ is summable. Then
\begin{equation}
\label{GaussGreen}
\nabla f(\Omega)=
\int_{\D^*\Omega}
\trN{f}{\nu}(x)\nu(x)\,d\Hn(x)
+
\int_{\capD \Omega}
\big(\trN{f}{\nu}(x)-\trN{f}{-\nu}(x)\big)\nu(x)\,d\Hn(x).
\end{equation}

\end{theorem}

\begin{proof}
It suffices to prove
\eqref{GaussGreen} only  for nonnegative functions
$f$. Indeed, to prove the theorem in the general case it suffices
to apply
 \eqref{GaussGreen} to $f^{+}$ and $f^{-}$ and then to use Corollary  \ref{cor_pos_neg}.

Obviously the right part of \eqref{GaussGreen} does not depend on a choice of
$\nu$. Note that if   $\trU{f}(x)\neq \trD{f}(x)$ then the normal to $E_t$
in the sense of Federer  at $x$ exists for all
$t\in (\trD{f}(x), \, \trU{f}(x))$ and does not depend on $t$.
Therefore we can suppose that at each such point $x$ the normal
$-\nu(x)$ coincides with the normal to $E_t$ in the sense of Federer for
$\trD{f}(x)<t<\trU{f}(x)$.
If we choose normals  $\nu$ in such a way, the formula
\eqref{GaussGreen} can be rewriten in the following form:
\begin{equation}
\label{GaussGreen_2}
\nabla f(\Omega)=
\int_{\D^*\Omega}
\trU{f}(x)\nu(x)\,d\Hn(x)
+
\int_{\capD \Omega}
\big(\trD{f}(x)-\trU{f}(x)\big)\nu(x)\,d\Hn(x).
\end{equation}

Obviously, if $P(E)<\infty$ then
$\nabla\chi_E(\Rn)=0$. By applying \eqref{Flem-Rish_formula} to the left part
of \eqref{GaussGreen_2} we obtain
\begin{equation*}
\begin{split}
\nabla f(\Omega)&=
\int_{0}^{\infty}\nabla\chi_{E_t}(\Omega)\,dt\\
&=-\int_{0}^{\infty}\nabla\chi_{E_t}(\Rn\setminus\Omega)\,dt
=-\int_{0}^{\infty}\nabla\chi_{E_t}(\Gamma\cap\D^*E_t)\,dt.
\end{split}
\end{equation*}
From the other hand, by \eqref{integr_normal} we get
$$
\nabla\chi_{E_t}(\Gamma\cap\D^*E_t)=
-\int_{\Gamma\cap\D^*E_t}\nu_{E_t}(x)\,d\Hn(x)=
-\int_{\Gamma\cap\D^*E_t}\nu(x)\,d\Hn(x),
$$
where $\nu_{E_t}$ is the normal to $E_t$ in the sense of Federer. Here the first equality
follows from the fact that
 $\nu_{E_t}(x)=\nu(x)$ for almost all $x\in\Gamma\cap\D^*E_t$,
and the latter equality is true since
 $\Hn(E_t\setminus\cup_{\tau>t} E_\tau)=0$
for almost all $t\in\R$.

Therefore, applying   \eqref{one_int_form} for $\eta=\nu$ we obtain
 \begin{equation*}
\begin{split}
\nabla f(\Omega)=&-\int_{0}^{+\infty}\nabla\chi_{E_t}(\Gamma\cap\D^*E_t)\,dt\\
=&\int_{0}^{+\infty}\int_{\Gamma\cap\D^*{E_t}}\nu(x)\,d\Hn(x)\\
=&\int_\Gamma \trU{f}(x)\nu(x)\,d\Hn(x)
+\int_{\capD\Omega}(\trU{f}(x)-\trD{f}(x))\nu(x)\,d\Hn(x).
\end{split}
\end{equation*}
The theorem is proved.
\end{proof}

\section{Average trace of a function in  $BV(\Omega)$}
Let $\Omega$ be a region with the countably rectifiable boundary $\D\Omega$.
Suppose that a function $f\in BV(\Omega)$ is summable in some neighborhood of a point
$x\in \Gamma$. Let us define the upper and lower average traces
of the function $f$ at  $x$ with respect to a normal  $\nu$ by equations:

\begin{align*}
\overline {f}(x,\nu)&=\limsup_{r\to 0}\, 2v_n^{-1}r^{-n}
\int_{B_r^\nu (x)} f(y)\,dy,\\
\underline {f}(x,\nu)&=\liminf_{r\to 0}\, 2v_n^{-1}r^{-n}
\int_{B_r^\nu (x)} f(x)\,dy.
\end{align*}

If $\overline {f}(x,\nu)=\underline {f}(x,\nu)$ then their common value
is called {\em average trace} and denoted
$\widetilde {f}(x,\nu)$.
First we prove some properties of average traces for nonnegative functions.

\begin{lemma}
\label{trace-lemma}
Suppose that a function $f\in BV(\Omega)$ is nonnegative and locally summable.
Then
$\underline f(x,\nu)\geq \trN{f}{\nu}(x)$.
\end{lemma}

\begin{proof} (Compare with the proof of Lemma 6.6.2/1 in \cite{M}.)

Lemma \ref{trace-lemma} is obviously true if  $\trN{f}{\nu}(x)=0$.
Suppose $0<\trN{f}{\nu}(x)$.
Pick $\ep >0$ and choose a number $t$ such that
$0<t<\trN{f}{\nu}(x)$ and $P_\Omega(E_t)<\infty$. Then $x\in \capN{\nu}{E_t}$.
This means that $\Theta^\nu_E(x)=1$. Therefore there exists $r_0(x)>0$ such that
$$
1-\ep<2v_n^{-n}r^{-n}\V(E_t\cap B_r^\nu (x))\leq1
$$
for $0<r<r_0(x)$.
Since
$$
\int_{B_r^\nu (x)}f(y)\,dy=\int_o^\infty \V(E_\tau\cap B_r^\nu (x))\,d\tau ,
$$
we obtain
\begin{align*}
2v_n^{-n}r^{-n}\V(B_r^\nu (x))\int_{B_r^\nu (x)}f(y)\,dy
&\geq
2v_n^{-n}r^{-n}
\int_o^t \V(E_\tau\cap B_r^\nu (x))\,d\tau\\
&\geq
2v_n^{-n}r^{-n}\V(E_\tau\cap B_r^\nu (x))t
\geq
(1-\ep)t.
\end{align*}
Since  $\ep$ is arbitrary we finish the proof by passing to the limit
as $r\to 0$, and then by passing to the limit as
$t\to \trN{f}{\nu}(x)$.

\end{proof}
\bigskip

\begin{theorem}
\label{average tr=tr}
If $f\in BV(\Omega)$ and  $\tracenorm{f}<\infty$ then the average trace
$\tilde f(x,\nu)$ of the function $f$ exists and equals to the trace
$\trN{f}{\nu}(x)$ almost everywhere on $\capN{\nu}\Omega$.
\end{theorem}
If the  function is bounded, the proof is unexpectedly simple.
\begin{lemma}
\label{bounded_trace}
Let a function $f\in BV(\Omega)$ be bounded. Then the average trace
$\tilde f(x,\nu)$ of the function $f$ exists almost everywhere on  $\Gamma$
and coincides with  $\trN{f}{\nu}(x)$.
\end{lemma}

\begin{proof}[Proof of the lemma.]
Let $|f|<C$.
By Lemma \ref{sum_trace} and the equation \ref{trace-lemma} it follows that
$$
\trN{f}{\nu}(x)=\trN{(f+C)}{\nu}(x)+\trN{(-C)}{\nu}(x)\leq\underline{(f+C)}(x,\nu)-C=
\underline{f}(x,\nu).
$$
Applying this inequality to $-f$, we obtain
$$
\trN{(-f)}{\nu}(x)\leq
\underline{(-f)}(x,\nu).
$$
Thus, by Lemma \ref{minus_minus} for almost all
$$
\trN{f}{\nu}(x)\geq
\overline{f}(x,\nu)
$$
for almost all $x\in\Gamma$.
The lemma is proved.
\end{proof}

\begin{proof}[Proof of Theorem \ref{average tr=tr}.]
As usual we may assume $f\geq0$.
Let us extend $f\in BV(\Omega)$ by zero to  $\Rn$.
By Lemma \ref{lem-const-out} the extended function  $f$ belongs to $BV(\Rn)$.
Suppose that a function $f\in\BVO$ is unbounded. Let us consider the set
$E=\{x\in\Omega \mid f(x)>0\}$ and show that
 $\tilde f(x,\nu)=0$ for almost all
$x\in \Gamma\setminus\capU{E}$.
Recall that almost all points of  $\D\Omega$ are located on  $C^1$-smooth $(n-1)$-dimensional
surfaces $M_i$ and a standard vector field $\nu$ is continuous along each
$M_i$. For a point $x\in \Gamma\setminus\capU{E}$ denote by $M$ just the surface $M_i$
such that $x\in M_i$.
For any point $p\in M$, the surface  $M$ divides a small ball centered at
 $p$ onto two open sets,
$U_1$ and $U_2$. Denote
$\tilde M=\D U_1\cap\D U_2\subset M$. It suffices to prove that
 $\bar f(x,\nu)=0$ at all points
$x\in \tilde M$ such that  $\Theta_x^\nu(E)=\Theta_x^{-\nu}(E)=0$.
For the sake of definiteness,  let the normals
$\nu$ are directed inward of $U_1$.

It is known that for $U_1$ and  $U_2$, the average  trace of each function
  $f\in BV(U_i)$, $i=1,\,2$, equals to its trace
(see \cite{M}, Theorem 6.6.2 or \cite{BM}, Lemma 13).
From the other hand, the  trace equals zero at almost all
$x\in M\setminus(\D^*(E\cap U_1)\cap \D^*(E\cap U_1))$. Therefore, for $i=1,2$
\begin{equation}
0=\lim_{r\to0}\frac{\int_{U_i\cap B_r(x)}f\, dx}{\V(U_i\cap B_r(x))}=
\lim_{r\to0}2v_n^{-1}r^{-n}\int_{U_i\cap B_r(x)}f\, dx.
\end{equation}
Thus
\begin{equation}
\limsup_{r\to0}2v_n^{-1}r^{-n}\int_{B^\nu_r(x)}f\, dx
\leq
\limsup_{r\to0}2v_n^{-1}r^{-n}\int_{B_r(x)}f\, dx=0.
\end{equation}
\medskip

Define
\begin{align}
f_C(x)=&
\begin{cases}
f(x)&\mbox{if } f(x)<C,\\
0&\mbox{if } f(x)\geq C,
\end{cases}&
f^C(x)=&
\begin{cases}
0&\mbox{if } f(x)<C,\\
f(x)&\mbox{if } f(x)\geq C.
\end{cases}
\end{align}

Now for almost all
$x\in \Gamma\setminus\capU{E_C}$ such that $0<f^\nu (x)<C$, we have
\begin{equation}
\overline{f}(x,\nu)=
\overline{f_C}(x,\nu)+
\overline{f^C}(x,\nu)=
\trN{(f_C)}{\nu}(x)+
\trN{(f^C)}{\nu}(x)=
\trN{f}{\nu}(x)+0.
\end{equation}

Taking into account that $\Hn(\cap_{t>0}\capU{E_t})=0$, we see that the theorem is proved.

\end{proof}

\end{document}